\documentclass[12pt]{article}
\usepackage{amsmath,amssymb}
\usepackage{tikz}
\usetikzlibrary{decorations.pathreplacing}

\usepackage{diagbox}

\usetikzlibrary{calc}
\usepackage{color}

\newtheorem{theorem}{Theorem}[section]
\newtheorem{lemma}[theorem]{Lemma}
\newtheorem{corollary}[theorem]{Corollary}
\newtheorem{proposition}[theorem]{Proposition}

\newtheorem{problem}[theorem]{Problem}
\newtheorem{question}[theorem]{Question}

\tikzstyle{noeud}=[circle,inner sep=2, minimum size =3 pt, line width = 1pt, draw=black, fill=white]

\newcommand{\proof}{\noindent{\bf Proof.\ }}
\newcommand{\qed}{\hfill $\square$ \bigskip}
\newcommand{\Dom}{{Dominator }}
\newcommand{\St}{{Staller }}

\def\cp{\,\square\,}

\newcommand{\tdom}{{\rm tdom}}
\newcommand{\cD}{{\cal D}}

\let\oldenumerate\enumerate
\renewcommand{\enumerate}{
  \oldenumerate
  \setlength{\itemsep}{0.5pt}
  \setlength{\parskip}{0pt}
  \setlength{\parsep}{0pt}
}

\begin{document}

\title{Maker-Breaker total domination game}

\author{
Valentin Gledel $^{a}$
\and
Michael A.\ Henning $^{b}$
\and
Vesna Ir\v si\v c $^{c,d}$
\and
Sandi Klav\v zar $^{c,d,e}$
}

\date{\today}

\maketitle
\begin{center}
$^a$ Univ Lyon, Universit\'e Lyon 1, LIRIS UMR CNRS 5205, F-69621, Lyon, France \\
{\tt valentin.gledel@univ-lyon1.fr}
\medskip

$^b$ Department of Pure and Applied Mathematics, University of Johannesburg,
Auckland Park 2006, South Africa \\
{\tt mahenning@uj.ac.za}
\medskip

$^c$ Faculty of Mathematics and Physics, University of Ljubljana, Slovenia\\
{\tt vesna.irsic@fmf.uni-lj.si}\\
{\tt sandi.klavzar@fmf.uni-lj.si}\\
\medskip

$^d$ Institute of Mathematics, Physics and Mechanics, Ljubljana, Slovenia\\
\medskip

$^e$ Faculty of Natural Sciences and Mathematics, University of Maribor, Slovenia
\end{center}

\begin{abstract}
Maker-Breaker total domination game in graphs is introduced as a natural counterpart to the Maker-Breaker domination game recently studied by Duch\^{e}ne, Gledel, Parreau, and Renault. Both games are instances of the combinatorial Maker-Breaker games. The Maker-Breaker total domination game is played on a graph $G$ by two players who alternately take turns choosing vertices of $G$. The first player, Dominator, selects a vertex in order to totally dominate $G$ while the other player, Staller, forbids a vertex to Dominator in order to prevent him to reach his goal. 

It is shown that there are infinitely many connected cubic graphs in which Staller wins and that no minimum degree condition is sufficient to guarantee that Dominator wins when Staller starts the game. An amalgamation lemma is established and used to determine the outcome of the game played on grids. Cacti are also classified with respect to the outcome of the game. A connection between the game and hypergraphs is established. It is proved that the game is PSPACE-complete on split and bipartite graphs. Several problems and questions are also posed.
\end{abstract}


\medskip\noindent
{\bf AMS Subj.\ Class:} 05C57, 05C69, 68Q25, 91A43

\section{Introduction}
\label{sec:intro}

The Maker-Breaker domination game ({\em MBD game} for short) was studied for the first time in~\cite{gledel-2018+}. The game is played on a graph $G$ by two players. To be consistent with the naming from the usual and well-investigated domination game~\cite{brklra-2010} (see also~\cite{bujtas-2015, dorbec-2015, heki-2016, kinnersley-2013, NSS2016, XLK-2018}), the players are named Dominator and Staller. They are selecting vertices alternatively, always selecting a vertex that has not yet been chosen. Dominator wins the MBD game on $G$ if at some point the set of vertices already selected by him forms a {\em dominating set} of $G$, that is, a set $D$ such that every vertex not in $D$ has a neighbor in $D$. Otherwise Staller wins, that is, she wins if she is able to select all the vertices from the closed neighborhood of some vertex.

Just as the total domination game~\cite{henning-2015} (see also~\cite{bresar-2017, bujtas-henning-tuza-2016, combinatorica-2017, hkr-2017+}) followed the domination game, we introduce here the Maker-Breaker total domination game ({\em MBTD game} for short). The rules of the  MBTD game are much the same as those of the MBD game, except that Dominator wins on $G$ if he can select a {\em total dominating set} of $G$, that is, a set $D$ such that every vertex of $G$ has a neighbor in $D$, and Staller wins if she can select all the vertices from the open neighborhood of some vertex.

The MBTD game can be, just as the MBD game, seen as a particular instance of the Maker-Breaker game introduced in 1973 by Erd\H{o}s and Selfridge~\cite{erdos-1973}. The game is played on a hypergraph $H$. One of the player, Maker, wins if he is able to select all the vertices of one of the hyperedges of $H$, while the other player, Breaker, wins if she is able to keep Maker from doing so. There is an abundant literature on this topic, see the books of Beck~\cite{beck-2008} and of Hefetz et al.~\cite{hefetz-2014} for related surveys. The MBTD game played on a graph $G$ can be interpreted as a Maker-Breaker game, where the hypergraph $H$ has the same vertices as $G$ and the hyperedges are the open neighborhoods of the vertices of $G$. Then Maker becomes \St and Breaker is renamed to \Dom. For more links between hypergraphs and the MBTD game see recent studies of the (total) domination game on hypergraphs~\cite{bujtas-2018, bujtas-henning-tuza-2016}.

Suppose the MBTD game is played on $G$. Then we say that the game is a {\em D-game} if Dominator is the first to play and it is an {\em S-game} otherwise. Whenever we say that the MBTD game is played on $G$, we mean that either the D-game or the S-game is played. If the D-game is played on a graph, then the sequence of moves of the two players will be denoted $d_1, s_1, d_2, s_2, \ldots$ Similarly, when the S-game is played, the sequence of moves will be denoted $s_1', d_1', s_2', d_2', \ldots$

We say that Staller {\em isolates a vertex} $u$ of $G$ during the MBTD game played on $G$ if she plays all of the neighbors of $u$ during the game. If so, then Staller wins the game on $G$.  A graph $G$ is
\begin{itemize}
\item ${\cal D}$, if Dominator wins the MBTD game;
\item ${\cal S}$, if Staller wins the MBTD game; and
\item ${\cal N}$, if the first player wins,
\end{itemize}
where it is assumed that both players are playing optimally.
We consider the empty graph to be a ${\cal D}$ graph because every vertex of it is dominated after zero moves have been played. Note also that $K_1$ is an ${\cal S}$ graph. The notations ${\cal D}$, ${\cal S}$, and ${\cal N}$ come directly from the article of Duch\^{e}ne et al.~\cite{gledel-2018+}, but are in turn derived from classical notations from combinatorial game theory (see~\cite{siegel-2013}).

We proceed as follows. In Section~\ref{sec:basic}, we derive basic properties of the game, and establish key lemmas that will be useful in subsequent chapters. We show in Section~\ref{sec:min-degree} that there exist ${\cal S}$ graphs with arbitrarily large minimal degree and that there are infinitely many examples of connected cubic graphs in which \St wins the S-game. An amalgamation lemma is established in Section~\ref{sec:grids} and applied to grid graphs to determine the outcome of the MBTD game. The concept of ${\cal D}$-minimal graphs is also discussed, in particular prisms over odd cycles are proved to be ${\cal D}$-minimal. Results on the MBTD game for cacti are presented in Section~\ref{sec:cacti}. Complexity results are discussed in Section~\ref{S:complexity} where we prove that deciding the outcome of the MBTD position is PSPACE-complete on split and bipartite graphs.  We close in Section~\ref{sec:concluding} with open problems and questions.

\section{Basic properties of the game}
\label{sec:basic}

The proof of the following result is parallel to the proof of~\cite[Proposition 2]{gledel-2018+}. The argument was given there only for completeness because it actually follows from the more general result \cite[Proposition 2.1.6]{hefetz-2014} dealing with arbitrary Maker-Breaker games on hypergraphs. We hence do not repeat the argument here.

\begin{lemma} {\rm (No-Skip Lemma)}
\label{lem:no-skip}
In an optimal strategy of Dominator (resp.\ Staller) in the MBTD game it is never an advantage for him (resp.\ for her) to skip a move.
\end{lemma}

No-Skip Lemma implies the following useful facts.

\begin{corollary}
\label{cor:win-in-the-other-game}
Let $G$ be a graph. 
\begin{itemize}
\item[(i)] If Dominator wins the S-game on $G$, then he also wins the D-game. If Staller wins the D-game, then she also wins the S-game.
\vspace*{-2mm}
\item[(ii)] Let $V_1, \ldots, V_k$ a partition of $V(G)$ such that $V_i$, $i\in [k] := \{1,\ldots, k\}$, induces a ${\cal D}$ graph, then $G$ is a ${\cal D}$ graph.
\vspace*{-2mm}
\end{itemize}
\end{corollary}

Another link between the two games is expressed by the following lemma.

\begin{lemma}
\label{lem_MBD-versus-MBTD}
If Staller wins an MBD game on a graph $G$, then Staller also wins an MBTD game on $G$. Equivalently, if Dominator wins an MBTD game on $G$, then Dominator also wins an MBD game on $G$.
\end{lemma}

\proof
A total dominating set is a dominating set. Therefore, if \Dom is able to select the vertices of a total dominating set of a graph $G$, then by applying the same strategy he is able to select the vertices of a dominating set of $G$. Likewise, if \St is able to keep \Dom from selecting a dominating set for $G$, then by applying the same strategy she can keep \Dom from selecting the vertices of a total dominating set.
\qed

Using Lemma~\ref{lem_MBD-versus-MBTD} we can thus apply earlier results on the MBD game, where \St has a winning strategy, to the the MBDT game.

\begin{proposition}
\label{prp:cycles}
The cycle $C_3$ is ${\cal N}$, the cycle $C_4$ is ${\cal D}$, and the cycles $C_n$, $n\ge 5$, are ${\cal S}$.
\end{proposition}

\proof
Let $v_1, \ldots, v_n$ be the vertices of $C_n$ in the natural order. The assertion for $C_3$ is clear. Consider the S-game on $C_4$ and assume without loss of generality that Staller played $v_1$ as the first vertex. Then Dominator replies with $v_3$ and wins the game in his next move. Hence $C_4$ is ${\cal D}$ by Corollary~\ref{cor:win-in-the-other-game}(i).

Let $n\ge 5$ and consider the D-game. We may again assume without loss of generality that Dominator played $v_1$ as the first vertex. Suppose first that $n = 5$. Then
Staller replies with the vertex $v_2$ which forces Dominator to play $v_5$. But then the move $v_4$ of Staller is her winning move. Assume second that $n\ge 6$. Then Staller plays the vertex $v_4$ which enables her to win after the next move by playing either $v_2$ or $v_6$. Hence $C_n$, $n\ge 5$, is ${\cal S}$ by Corollary~\ref{cor:win-in-the-other-game}(i).
\qed

Let $G$ and $H$ be disjoint graphs and consider the MBTD game played on the disjoint union of $G$ and $H$. In Table~\ref{table:table_union}, all possible outcomes of the game are presented. Table~\ref{table:table_union} is identical to the corresponding table from~\cite{gledel-2018+} for the MBD game. The entries ${\cal S}$ from our table follow from their table by applying Lemma~\ref{lem_MBD-versus-MBTD}. For the other three entries the arguments are parallel to those from~\cite{gledel-2018+} and are skipped here.

\begin{table}[ht!]
\begin{center}
\begin{tabular}{|c||
c|c|c|}
 \hline
\diagbox{$G$}{$H$}  & ${\cal D}$ & ${\cal N}$ & ${\cal S}$ \\
 \hline
 \hline
 ${\cal D}$ & ${\cal D}$ & ${\cal N}$  & ${\cal S}$ \\ \hline
 ${\cal N}$ & ${\cal N}$ & ${\cal S}$ & ${\cal S}$ \\ \hline
 ${\cal S}$ & ${\cal S}$ & ${\cal S}$ & ${\cal S}$ \\ \hline
 \end{tabular}
\end{center}
\caption{Outcomes of the MBTD game played on the disjoint union of $G$ and $H$}
\label{table:table_union}
\end{table}

If $u$ is a vertex of a graph $G$ and $H$ is an arbitrary graph disjoint from $G$, then let $G_u[H]$ be the graph constructed from $G$ by replacing the vertex $u$ with $H$ and joining with an edge every vertex of $H$ to every vertex from $N_G(u)$. Note that $G_u[K_1] = G$, where $u$ is an arbitrary vertex of $G$, and that $(K_2)_w[\overline{K}_k] = K_{1,k}$, where $w$ is an arbitrary vertex of $K_2$.

\begin{lemma} {\rm (Blow-Up Lemma)}
\label{lem:blow-up}
Let $u$ be a vertex of a graph $G$ and let $H$ be a graph. If $G$ is ${\cal D}$, then also $G_u[H]$ is ${\cal D}$.
\end{lemma}

\proof
Suppose that $G\in {\cal D}$ and let the MBTD game be played on $G_u[H]$. Then the strategy of Dominator is the following. He imagines a game is played on $G$, where he is playing optimally. Each move of Staller played in the game on $G_u[H]$ is copied by Dominator to the imagined game on $G$, provided the move is legal in $G$. Dominator then replies with his optimal move and copies it to the real game played on $G_u[H]$.  More precisely, whenever his optimal move in $G$ is a vertex $w\ne u$, he also plays $w$ in $G_u[H]$, while if $u$ is his optimal move in the game played on $G$, then in the game on $G_u[H]$ he plays an arbitrary vertex of $H$. Suppose now that a move of Staller played in $G_u[H]$ is not legal in the game on $G$. This happens when Staller plays a vertex of $H$. If $u$ is a legal move of Staller in $G$, then Dominator imagines that she has played $u$ in $G$. Otherwise, he imagines that Staller has skipped her move in the imagined game played on $G$. In both cases Dominator then replies with his optimal move. In any case, having in mind the No-Skip Lemma, Dominator wins the imagined game played on $G$.

Suppose that $D$ is the total dominating set selected by Dominator when the game played on $G$ is finished. If $u\notin D$, then $D$ is also the set of vertices selected by Dominator in the real game. Moreover, $D$ is also a total dominating set of $G_u[H]$. Indeed, a vertex $w\in D$ which (totally) dominates $u$ in $G$, (totally) dominates every vertex of $H$ in $G_u[H]$. Suppose next that $u\in D$ and let $x$ be the vertex of $H$ selected by Dominator in $G_u[H]$, when he selected $u$ is the imagined game. Hence in the real game Dominator selected the set $D' = D\cup \{x\}\setminus \{u\}$. Since $D'$ is a total dominating set of $G_u[H]$, Dominator wins the real game also in this case.
\qed

To see that the converse of the Blow-Up Lemma~\ref{lem:blow-up} does not hold, consider $(C_5)_u[C_4]$, where $u$ is an arbitrary vertex of $C_5$. From Proposition~\ref{prp:cycles} we know that $C_5$ is ${\cal S}$ (and so not ${\cal D}$). On the other hand it can be easily verified that $(C_5)_u[C_4]$ is ${\cal D}$.

The lexicographic product $G[H]$ of graphs $G$ and $H$ has vertex set $V(G) \times V(H)$, where vertices $(g,h)$ and $(g',h')$ are adjacent if $gg' \in E(G)$, or if $g=g'$ and $hh' \in E(H)$. Iteratively applying the Blow-Up Lemma to all the vertices of $G$, we get the following consequence:

\begin{corollary}
If $G\in {\cal D}$ and $H$ is a graph, then $G[H]\in {\cal D}$.
\end{corollary}

\section{Graphs from ${\cal S}$ and ${\cal N}$ with large minimum degree}
\label{sec:min-degree}

In this section we consider the effect of the minimum degree of a graph to the outcome of the MBTD game. An intuition says that when the minimum degree is large, then Dominator has good chances to win the game. We will demonstrate here, however, that this is not true in general.

The following construction shows that there exist ${\cal S}$ graphs with arbitrarily large minimal degree. Let $G_{n, k}$, $k\ge 1$, $n \geq 2k$, be the graph with the vertex set $[n] \cup \binom{[n]}{k}$ and edges between $i \in [n]$ and $S \in \binom{[n]}{k}$ if and only if $i \in S$. In particular, $G_{2,1} = 2K_2$. The graph $G_{n,k}$ is bipartite with vertices of degrees $k$ and $\binom{n-1}{k-1}$, hence $\delta(G_{n, k}) = k$. Note that $\gamma_t(G_{n, k}) = \lceil \frac{n}{k} \rceil + n - k + 1$. Since $n\ge 2k$, Staller can play some $k$ vertices $i_1, \ldots, i_k \in [n]$ in the D-game on $G_{n, k}$ before either Dominator wins the game or there are no more legal moves in $[n]$. Hence Staller wins the D-game on $G_{n, k}$ as she isolates the vertex $\{i_1, \ldots, i_k\} \in \binom{[n]}{k}$. As she wins the D-game, she also wins the S-game. Thus $G_{n, k}$ is ${\cal S}$.

The \emph{total domatic number} of a graph $G$, denoted by $\tdom(G)$ and first defined by Cockayne Dawes, and Hedetniemi~\cite{CDH80total}, is the maximum number of total dominating sets into which the vertex set of $G$ can be partitioned. The parameter $\tdom(G)$ is equivalent to the maximum number of colors in a (not necessarily proper) coloring of the vertices of a graph where every color appears in every open neighborhood. Chen, Kim, Tait, and Verstraete~\cite{CKTV15coupon} called this the \emph{coupon coloring problem}. This parameter is now well studied. We refer the reader to Chapter~13 in the book~\cite{MHAYbookTD} on total domination in graphs for a brief survey of results on the  total domatic number, and to~\cite{GoHe18} for a recent paper on this topic.

\begin{theorem}
\label{t:tdom}
If $G$ is a graph with $\tdom(G) = 1$, then the S-game is won by Staller.
\end{theorem}
\proof If the D-game played in $G$ is won by Staller, then Staller also wins the S-game. Hence we may assume that the D-game played in $G$ is won by Dominator, for otherwise the desired result is immediate. Thus, \Dom has a winning strategy to construct a total dominating set whatever sequence of moves \St plays in order to prevent him from doing so. In this case, in the S-game with \St first to play, she applies exactly Dominator's winning strategy in the D-game played in $G$ in order to create a total dominating set, say $D$, in $G$ comprising of the vertices she plays in the course of the game. All Dominator's moves are played from the set $V(G) \setminus D$.  By assumption, $\tdom(G) = 1$, implying that the set $V(G) \setminus D$ is not a total dominating set of $G$. Thus, the vertices played by \Dom do not form a total dominating set of $G$, implying that the S-game is won by Staller.~\qed

\medskip
We observe that every tree $G$ satisfies $\tdom(G) = 1$, and so by Theorem~\ref{t:tdom}, the S-game is won by \St in the class of trees. As observed in~\cite{DeHaHe17}, there are infinitely many examples of connected cubic graphs $G$ with $\tdom(G) = 1$. The Heawood graph, $G_{14}$, shown in Figure~\ref{f:Heawood}(a) is one such example of a cubic graph that does not have two disjoint total dominating sets; that is, $\tdom(G_{14}) = 1$. Zelinka~\cite{Ze89} was the first who observed that that  the graphs $G_{n,k}$, $n\ge 2k-1$, from the beginning of the section satisfy $\tdom(G_{n,k}) = 1$ (and have arbitrarily large minimum degree). We summarize these results formally as follows.

\begin{corollary}
The following holds. \vspace*{-2mm}
\begin{enumerate}
\item There are infinitely many examples of connected cubic graphs in which \St wins the S-game.
\item No minimum degree condition is sufficient to guarantee that \Dom wins the S-game.
\end{enumerate}
\end{corollary}

Consider the graphs $G_{2k-1,k}$, $k\ge 2$. As observed earlier, $\tdom(G_{2k-1,k}) = 1$ and hence Staller wins the S-game. On the other hand, Dominator wins the D-game. The main idea of Dominator is to select $k$ vertices from the set $[2k-1]$ and two vertices among the other vertices, so that these $k+2$ vertices form a total dominating set. This goal can be easily achieved by following Staller's moves in the bipartition parts of $G_{2k-1,k}$.

We next demonstrate that the Petersen graph $P$ is ${\cal S}$. This example is in particular interesting because $\tdom(P) = 2$. In the S-game, due to the symmetries of the graph, $s'_1$ can be an arbitrary vertex, and Dominator's reply can be on a neighbor of $s'_1$ or on a vertex at distance $2$ from $s'_1$. After Dominator's first move, Staller has a strategy which forces Dominator's replies (otherwise she wins the game even sooner) and in both cases, the strategy ends with Staller isolating a vertex. The strategies are presented on Figure~\ref{fig:petersenS}, where in both cases, she wins by isolating one of the two black vertices.

\begin{figure}[!ht]
    \begin{center}
        \begin{tikzpicture}[scale=0.4]
        \tikzstyle{every node}=[circle, draw, fill=black!10,
        inner sep=0pt, minimum width=4pt]

        \begin{scope}
        \node (0) at (0*360/5:4 cm) {};
        \node (1) at (1*360/5:4 cm) {};
        \node (2) at (2*360/5:4 cm) {};
        \node (3) at (3*360/5:4 cm) {};
        \node (4) at (4*360/5:4 cm) {};
        \draw (0) -- (1) -- (2) -- (3) -- (4) -- (0);

        \node (5) at (0*360/5:2 cm) {};
        \node (6) at (1*360/5:2 cm) {};
        \node (7) at (2*360/5:2 cm) {};
        \node (8) at (3*360/5:2 cm) {};
        \node (9) at (4*360/5:2 cm) {};

        \draw (5) edge (7);
        \path (7) edge (9);
        \path (9) edge (6);
        \path (6) edge (8);
        \path (8) edge (5);

        \path (0) edge (5);
        \path (1) edge (6);
        \path (2) edge (7);
        \path (3) edge (8);
        \path (4) edge (9);

        \node[label={above:$s_1'$}] () at (1*360/5:4 cm) {};
        \node[label={left:$d_1'$}, fill=black] () at (2*360/5:4 cm) {};
        \node[label={below:$s_2'$}] () at (3*360/5:2 cm) {};
        \node[label={right:$d_2'$}] () at (4*360/5:2 cm) {};
        \node[label={above:$s_3'$}] () at (2*360/5:2 cm) {};
        \node[fill=black] () at (0*360/5:2 cm) {};
        \end{scope}

        \begin{scope}[xshift=12cm]
        \node (0) at (0*360/5:4 cm) {};
        \node (1) at (1*360/5:4 cm) {};
        \node (2) at (2*360/5:4 cm) {};
        \node (3) at (3*360/5:4 cm) {};
        \node (4) at (4*360/5:4 cm) {};
        \draw (0) -- (1) -- (2) -- (3) -- (4) -- (0);

        \node (5) at (0*360/5:2 cm) {};
        \node (6) at (1*360/5:2 cm) {};
        \node (7) at (2*360/5:2 cm) {};
        \node (8) at (3*360/5:2 cm) {};
        \node (9) at (4*360/5:2 cm) {};

        \draw (5) edge (7);
        \path (7) edge (9);
        \path (9) edge (6);
        \path (6) edge (8);
        \path (8) edge (5);

        \path (0) edge (5);
        \path (1) edge (6);
        \path (2) edge (7);
        \path (3) edge (8);
        \path (4) edge (9);

        \node[label={above:$s_1'$}] () at (1*360/5:4 cm) {};
        \node[label={left:$d_1'$}, fill=black] () at (3*360/5:4 cm) {};
        \node[label={below:$s_2'$}] () at (3*360/5:2 cm) {};
        \node[label={right:$d_2'$}] () at (4*360/5:2 cm) {};
        \node[label={below:$s_3'$}] () at (4*360/5:4 cm) {};
        \node[fill=black] () at (0*360/5:4 cm) {};
        \end{scope}

        \end{tikzpicture}
        \caption{The strategy of Staller in the S-game on $P$. }
        \label{fig:petersenS}
    \end{center}
\end{figure}
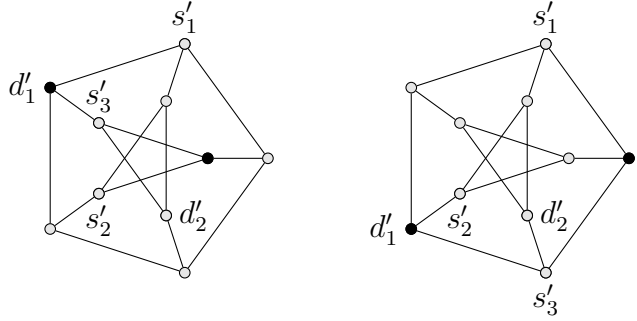

In the D-game, Staller's strategy is the following. She replies to $d_1$ on its neighbor. Up to symmetries there are only three possible replies $d_2$ of Dominator (at distances $1$ or $2$ from $d_1$, $s_1$). Staller's strategies in each of these cases are presented in Figure~\ref{fig:petersenD}  where she wins by isolating one of the two black vertices in each case. Hence, Staller also wins the D-game.

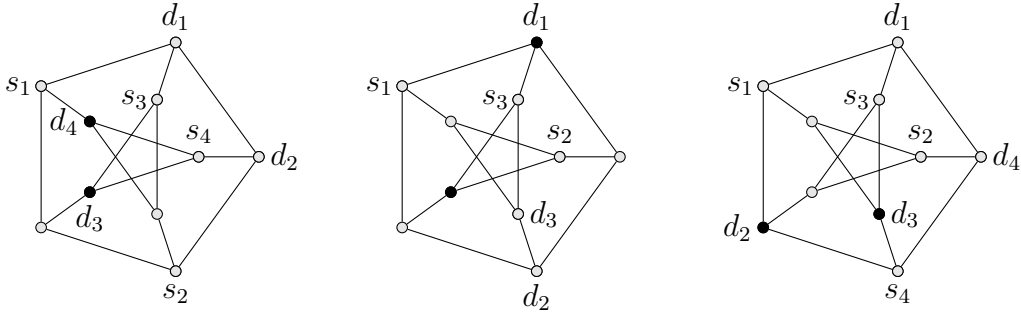
\begin{figure}[!ht]
	\begin{center}
		\begin{tikzpicture}[scale=0.4]
		\tikzstyle{every node}=[circle, draw, fill=black!10,
		inner sep=0pt, minimum width=4pt]
		
		\begin{scope}
		\node (0) at (0*360/5:4 cm) {};
		\node (1) at (1*360/5:4 cm) {};
		\node (2) at (2*360/5:4 cm) {};
		\node (3) at (3*360/5:4 cm) {};
		\node (4) at (4*360/5:4 cm) {};
		\draw (0) -- (1) -- (2) -- (3) -- (4) -- (0);
		
		\node (5) at (0*360/5:2 cm) {};
		\node (6) at (1*360/5:2 cm) {};
		\node (7) at (2*360/5:2 cm) {};
		\node (8) at (3*360/5:2 cm) {};
		\node (9) at (4*360/5:2 cm) {};
		
		\draw (5) edge (7);
		\path (7) edge (9);
		\path (9) edge (6);
		\path (6) edge (8);
		\path (8) edge (5);
		
		\path (0) edge (5);
		\path (1) edge (6);
		\path (2) edge (7);
		\path (3) edge (8);
		\path (4) edge (9);
		
		\node[label={above:$d_1$}] () at (1*360/5:4 cm) {};
		\node[label={left:$s_1$}] () at (2*360/5:4 cm) {};
		\node[label={below:$d_3$}, fill=black] () at (3*360/5:2 cm) {};
		\node[label={left:$s_3$}] () at (1*360/5:2 cm) {};
		\node[label={left:$d_4$}, fill=black] () at (2*360/5:2 cm) {};
		\node[label={above:$s_4$}] () at (5*360/5:2 cm) {};
		\node[label={right:$d_2$}] () at (0*360/5:4 cm) {};
		\node[label={below:$s_2$}] () at (4*360/5:4 cm) {};
		\end{scope}
		
		\begin{scope}[xshift=12cm]
		\node (0) at (0*360/5:4 cm) {};
		\node (1) at (1*360/5:4 cm) {};
		\node (2) at (2*360/5:4 cm) {};
		\node (3) at (3*360/5:4 cm) {};
		\node (4) at (4*360/5:4 cm) {};
		\draw (0) -- (1) -- (2) -- (3) -- (4) -- (0);
		
		\node (5) at (0*360/5:2 cm) {};
		\node (6) at (1*360/5:2 cm) {};
		\node (7) at (2*360/5:2 cm) {};
		\node (8) at (3*360/5:2 cm) {};
		\node (9) at (4*360/5:2 cm) {};
		
		\draw (5) edge (7);
		\path (7) edge (9);
		\path (9) edge (6);
		\path (6) edge (8);
		\path (8) edge (5);
		
		\path (0) edge (5);
		\path (1) edge (6);
		\path (2) edge (7);
		\path (3) edge (8);
		\path (4) edge (9);
		
		\node[label={above:$d_1$}, fill=black] () at (1*360/5:4 cm) {};
		\node[label={left:$s_1$}] () at (2*360/5:4 cm) {};
		\node[label={right:$d_3$}] () at (4*360/5:2 cm) {};
		\node[label={left:$s_3$}] () at (1*360/5:2 cm) {};
		\node[label={above:$s_2$}] () at (5*360/5:2 cm) {};
		\node[label={below:$d_2$}] () at (4*360/5:4 cm) {};
		\node[fill=black] () at (3*360/5:2 cm) {};
		\end{scope}
		
		\begin{scope}[xshift=24cm]
		\node (0) at (0*360/5:4 cm) {};
		\node (1) at (1*360/5:4 cm) {};
		\node (2) at (2*360/5:4 cm) {};
		\node (3) at (3*360/5:4 cm) {};
		\node (4) at (4*360/5:4 cm) {};
		\draw (0) -- (1) -- (2) -- (3) -- (4) -- (0);
		
		\node (5) at (0*360/5:2 cm) {};
		\node (6) at (1*360/5:2 cm) {};
		\node (7) at (2*360/5:2 cm) {};
		\node (8) at (3*360/5:2 cm) {};
		\node (9) at (4*360/5:2 cm) {};
		
		\draw (5) edge (7);
		\path (7) edge (9);
		\path (9) edge (6);
		\path (6) edge (8);
		\path (8) edge (5);
		
		\path (0) edge (5);
		\path (1) edge (6);
		\path (2) edge (7);
		\path (3) edge (8);
		\path (4) edge (9);
		
		\node[label={above:$d_1$}] () at (1*360/5:4 cm) {};
		\node[label={left:$s_1$}] () at (2*360/5:4 cm) {};
		\node[label={right:$d_3$}, fill=black] () at (4*360/5:2 cm) {};
		\node[label={left:$s_3$}] () at (1*360/5:2 cm) {};
		\node[label={above:$s_2$}] () at (5*360/5:2 cm) {};
		\node[label={right:$d_4$}] () at (0*360/5:4 cm) {};
		\node[label={below:$s_4$}] () at (4*360/5:4 cm) {};
		\node[label={left:$d_2$}, fill=black] () at (3*360/5:4 cm) {};
		\end{scope}
		
		\end{tikzpicture}
		\caption{The strategy of Staller in the D-game on $P$.}
		\label{fig:petersenD}
	\end{center}
\end{figure}

\section{Amalgamation lemma and grids}
\label{sec:grids}

In this section we solve the MBTD game for grids and some Cartesian products of paths and cycles. No similar results are known for the (total) domination game. To derive the results the following amalgamation lemma will be utmost useful.

Let $G$ and $H$ be disjoint graphs, and let $G'$ and $H'$ be subgraphs of $G$ and $H$, respectively. If $G'$ and $H'$ are isomorphic, then the {\em amalgamation} of $G$ and $H$ over $G'=H'$ is the graph obtained from the disjoint union of $G$ and $H$ by identifying $G'$ with $H'$ (w.r.t.\ a given, fixed isomorphism $G'\rightarrow H'$).

\begin{lemma}
\label{lem:amalgamation}
Let $G$ be the amalgam of $G_1$ and $G_2$ over $G_0 = G_1\cap G_2$. If Staller has a winning strategy on $G_1$ such that for every possible sequence of moves she isolates a vertex from $G_1\setminus G_2$, then she also has a winning strategy on $G$.
\end{lemma}

\proof
Suppose that Staller has a winning strategy on $G_1$ as stated and let the MBTD game be played on $G$. The strategy of Staller is to replicate her specified strategy from $G_1$ on the game on $G$. If Dominator always plays on $G_1$, then Staller clearly wins. On the other hand, if Dominator plays some vertices of $G_2\setminus G_1$, then by Lemma~\ref{lem:no-skip} and the assumption that for every possible sequence of moves she isolates a vertex from $G_1\setminus G_2$, Staller also wins.
\qed

As an application of Lemma~\ref{lem:amalgamation} we prove the following result.

\begin{theorem}
\label{thm:grids}
If $n,m\ge 2$, then the MBTD game played on $P_m\cp P_n$ is ${\cal D}$ if both $n$ and $m$ are even, and it is ${\cal S}$ otherwise.
\end{theorem}

\proof
Consider first the grids $P_{2k}\cp P_{2\ell}$. Then $V(P_{2k}\cp P_{2\ell})$ can be partitioned into sets of order $4$ each inducing a $4$-cycle. Since $C_4$ is a ${\cal D}$ graph (Proposition~\ref{prp:cycles}), Corollary~\ref{cor:win-in-the-other-game}(ii) implies that $P_{2k}\cp P_{2\ell}$ is a ${\cal D}$ graph.

Consider next the prisms $P_{2}\cp P_{2\ell+1}$, $\ell\ge 1$. Setting $V(P_n) = [n]$ we have $V(P_{2}\cp P_{2\ell+1}) = \{(i,j):\ i\in [2], j\in [2\ell +1]\}$. Let $X$ and $Y$ be the bipartition sets of $P_{2}\cp P_{2\ell+1}$, where $(1,1)\in X$. Assume without loss of generality that in the D-game Dominator first played a vertex from $X$. Staller replies with the move $(1,2)\in Y$. Then Dominator is forced to play $(2,1)$ in order to totally dominate the vertex $(1,1)$, for otherwise Staller would already win the game in her next move. Inductively, if $(1,2q)$ is the last vertex played by Staller, her next move is $(1,2q+2)\in Y$ and this forces Dominator to play $(2,2q+1)\in Y$. Since the first move of Dominator was a vertex from $X$, all these moves are legal throughout the game and no threats of Staller are predominated. When Staller finally plays $(1,2\ell)$, one of the vertices $(1,2\ell-1)$ and $(1,2\ell+1)$ can be isolated by Staller after the next move of Dominator. Hence Staller will win the game and thus $P_{2}\cp P_{2\ell+1}$ is ${\cal S}$ by Corollary~\ref{cor:win-in-the-other-game}(i).

Note that by the above strategy of Staller on the game played on $P_{2}\cp P_{2\ell+1}$, the set of vertices on which she can finish the game by isolating them is a subset of $[1]\times [2\ell +1]$. Consider next the grids $P_{m}\cp P_{2\ell+1}$, $m\ge 4$. Then we may without loss of generality assume that in the D-game Dominator first plays a vertex $(i,j)$, where $i\ge 3$. Consider now $P_{m}\cp P_{2\ell+1}$ as the amalgamation of the prism $P_{2}\cp P_{2\ell+1}$ induced by the vertices $[2]\times [2\ell +1]$ and the prism $P_{m-1}\cp P_{2\ell+1}$ induced by the vertices $\{2,\ldots, m\}\times [2\ell +1]$. Then by the above and by Lemma~\ref{lem:amalgamation}, Staller has a winning strategy.

It remains to consider the grid $P_{3}\cp P_{3}$. If Dominator starts with $(2,2)$, then Staller replies with $(1,2)$ threatening $(1,1)$ and $(1,3)$. Otherwise, assume without loss of generality that Dominator first plays a vertex $(3,j)$, where $j\in [3]$. Then Staller replies with the move $(1,2)$, threatening the same vertices as before.
\qed

\begin{theorem}
\label{t:P2kcycle}
For $k \ge 1$ and $m \ge 3$, the MBTD game played on $P_{2k}  \cp C_{m}$ is $\cD$.
\end{theorem}
\proof We first consider the MBTD game played on $P_{2k} \cp C_{2\ell}$ for some integers $k, \ell \ge 1$. Since $V(P_{2k}  \cp C_{2\ell})$ can be partitioned into sets of order~$4$ each inducing a $4$-cycle, it follows from Corollary~\ref{cor:win-in-the-other-game}(ii) Proposition~\ref{prp:cycles} that $P_{2k} \cp C_{2\ell}$ is a $\cD$ graph.

Next we consider the prism $P_{2} \cp C_{2\ell + 1}$ for some integer $\ell \ge 1$.  Setting $V(P_2) = [2]$ and $V(C_n) = [n]$, we have $V(P_2 \cp C_{2\ell + 1}) = \{(i,j):\ i \in [2], j \in [2\ell+1] \}$. \Dom now imagines he is playing on the imaginary $(4\ell + 2)$-cycle $C$ given by $v_1 v_2 \ldots v_{4\ell + 2} v_1$ where
\[
v_i = \left\{
\begin{array}{ll}
(1,i); & \mbox{if $i \in [2\ell+1]$ is odd}\,, \\
(2,i); & \mbox{if $i \in [2\ell]$ is even} \\
\end{array}
\right.
\]
and
\[
v_{(2\ell + 1) + i} = \left\{
\begin{array}{ll}
(2,i); & \mbox{if $i \in [2\ell+1]$ is odd}\,, \\
(1,i); & \mbox{if $i \in [2\ell]$ is even}\,. \\
\end{array}
\right.
\]
That is, $C \colon v_1 v_2 \ldots v_{4\ell + 2} v_1$ is the cycle is given by
\[
(1,1), (2,2), (1,3), (2,4), \ldots, (1,2\ell + 1), (2,1), (1,2), (2,3),\ldots, (2,2\ell + 1), (1,1)\,.
\]

We note that every vertex in the prism $P_{2}  \cp C_{2\ell + 1}$ has exactly three neighbors and these three neighbors appear as three consecutive vertices on the cycle $C$. Thus, Dominator's strategy is to guarantee that no three consecutive vertices on the (imaginary) cycle $C$ are all played by Staller. \Dom achieves his goal as follows. Suppose that Staller's first move of the game played on $P_{2}  \cp C_{2\ell + 1}$ is the vertex $v_i$ for some $i \in [4\ell + 2]$. \Dom responds as follows. \Dom plays the vertex $v_{i+1}$ if it has not yet been played (where addition is taken modulo~$2\ell + 4$). If, however, the vertex $v_{i+1}$ has already been played, then \Dom plays the vertex $v_{i-1}$ if it has not yet been played. Otherwise, if both $v_{i+1}$ and $v_{i-1}$ have already been played in the game, then \Dom simply plays an arbitrary (legal) vertex that has not yet been played.

Suppose, to the contrary, that there are three consecutive vertices $v_i, v_{i+1}, v_{i+2}$ all played by \St during the course of the game. If \St played the vertex $v_i$ before $v_{i+1}$, then \Dom would have replied to her move $v_i$ by playing $v_{i+1}$, contradicting our supposition that $v_{i+1}$ is played by Staller. Hence, \St played the vertex $v_{i+1}$ before she played the vertex $v_{i}$. Dominator's strategy implies that when \St played the vertex $v_{i+1}$, he would have either played the vertex $v_{i+2}$ if it had not yet been played or he would play the vertex $v_i$ (which has not yet been played). In the former case, we contradict the supposition that $v_{i+2}$ is played by Staller. In the latter case, we contradict the supposition that $v_{i}$ is played by Staller. Since both cases produce a contradiction, we deduce that no three consecutive vertices on $C$ are all played by \St during the course of the game. This implies by our earlier observations, that \Dom plays a neighbor of every vertex in the prism $P_{2} \cp C_{2\ell + 1}$. Equivalently, the moves played by \Dom form a total dominating set in the prism $P_{2} \cp C_{2\ell + 1}$. Hence, \Dom wins the game, as claimed.

Next we consider the prism $P_{2k}  \cp C_{2\ell + 1}$ for some integer $k \ge 2$ and $\ell \ge 1$.  We note that there is a partition of $V(P_{2k} \cp C_{2\ell + 1})$ into sets $V_1, V_2, \ldots, V_k$ each of which induce a copy of the prism $P_{2} \cp C_{2\ell + 1}$. Since the prism $P_{2}  \cp C_{2\ell + 1}$ is a $\cD$ graph, Corollary~\ref{cor:win-in-the-other-game}(ii) implies that $P_{2k}  \cp C_{2\ell + 1}$ is a $\cD$ graph.~\qed

\begin{theorem}
\label{t:P3cycle}
If $k \ge 3$ and $k \ne 4$, then \St wins the S-game played on $P_{3}  \cp C_{k}$.
\end{theorem}
\proof We first consider the S-game played on  $P_{3}  \cp C_{3}$. Letting $V(P_3) = [3]$ and $V(C_3) = [3]$, we have $V(P_3  \cp C_{3}) = \{(i,j):\ i, j \in [3]\}$. \St first plays the vertex $s_1'= (2,2)$. By symmetry, we may assume that the move $d_1'$ is one of the vertices $(1,1)$, $(1,2)$ or $(2,3)$.

Suppose that $d_1'\in \{(1,1), (1,2)\}$. In this case, \St plays $s_2' = (3,3)$, thereby forcing \Dom to play $d_2' = (3,1)$. \St then plays $s_3' = (2,1)$ with the double threat of playing $(1,3)$ and $(3,2)$.

Suppose next that $d_1' = (2,1)$. In this case, \St plays $s_2' = (3,1)$, thereby forcing \Dom to play $d_2' = (3,3)$. \St then continues with $s_3' = (1,1)$ with the double threat of playing $(1,3)$ and $(2,3)$. In both cases, Staller has a winning strategy. Thus, \St wins the S-game played on $P_{3}  \cp C_{3}$.

Next we consider the S-game played on $P_3  \cp C_{k}$ where $k \ge 5$. Letting $V(P_3) = [3]$ and $V(C_k) = [k]$, we have $V(P_3 \cp C_{k}) = \{(i,j):\ i \in [3], j \in [k] \}$. \St starts with an arbitrary vertex of degree~$4$; that is, she plays $s_1' = (2,i)$ for some $i \in [k]$. By symmetry and for notational convenience, we may assume that $s_1' = (2,3)$. By symmetry, for $i \in [3]$ and $j \in [2]$, the moves $d_1' = (i,j)$ and $d_1' = (i,6-j)$ are identical. Further, the  moves $d_1' = (1,4)$ and $d_1' = (3,4)$ are identical, as are the first moves $d_1' = (1,k)$ and $d_1' = (3,k)$. Hence we may assume that $d_1'\ne (i,j)$, where $i \in [3]$ and $j \in [2]$ and that $d_1' \notin \{(3,4), (3,k)\}$. With these assumptions, \St plays $s_2' = (3,2)$, thereby forcing \Dom to play $d_2' = (3,4)$. \St now replies with $s_3' = (2,1)$ with the double threat of playing $(1,2)$ and $(3,k)$. Thus \St has a winning strategy, implying that she wins the S-game played on $P_3  \cp C_{k}$ for $k \ge 5$.~\qed

To conclude the section, let us call a connected graph $G$ to be ${\cal D}$-minimal if $G$ is ${\cal D}$ but $G-e$ is not ${\cal D}$ for an arbitrary $e\in E(G)$. Then we have:

\begin{proposition}
\label{p:Dminimal}
The following holds. \vspace*{-2mm}
\begin{enumerate}
\item If $k \ge 2$, then $K_{2,k}$ is $\cD$-minimal.
\item If $k\ge 1$, then $P_2 \cp C_{2k+1}$ is $\cD$-minimal.
\end{enumerate}
\end{proposition}
\proof By Proposition~\ref{prp:cycles}, the cycle $C_4$ is $\cD$; that is, $K_{2,2}$ is $\cD$. Applying the Blow-Up Lemma to the graph $G = C_4$ and to an arbitrary vertex $u$ of $G$, shows that the graph $G_u[\overline{K}_{k-1}] \cong K_{2,k}$ is $\cD$ for every $k \ge 3$. ($\overline{G}$ denotes the complement of a graph $G$.) Thus, $K_{2,k}$ is $\cD$. Further, if $G \cong K_{2,k}$ and $e$ is an arbitrary edge of $G$, then $G - e$ contains a vertex of degree~$1$. \St wins the S-game on $G - e$ by selecting $s_1'$ as the neighbor of the vertex of degree~$1$. Hence, $G - e$ is not $\cD$, implying that $K_{2,k}$ is $\cD$-minimal. This proves Part~(a).

To prove Part~(b), consider the graph $G = P_2 \cp C_{2k+1}$ where $k \ge 1$. By Theorem~\ref{t:P2kcycle}, the MBTD game played on $G$ is $\cD$. We now consider an arbitrary edge $e$ of $G$. Setting $V(P_2) = [2]$ and $V(C_n) = [n]$, we have $V(G) = \{(i,j):\ i \in [2], j \in [2k+1] \}$. By symmetry we may assume that $e$ is either the edge joining the vertices $(1,2k)$ and $(2,2k)$ or the edge joining the vertices $(1,1)$ and $(1,2k+1)$.

Consider the S-game played on $G$. Suppose firstly that $e = (1,2k)(2,2k)$. \St plays $s_1' = (1,2k-1)$, thereby forcing \Dom to play $d_1' = (1,2k+1)$. If $k=1$, then \St plays $s_2' = (2,3)$ with the double threat of playing the vertex $(1,2)$ and the vertex $(2,1)$. Hence, we may assume that $k \ge 2$, for otherwise \St immediately wins the S-game in $G - e$. With this assumption, \St plays as her second move the vertex $(2,2k-1)$, thereby forcing \Dom to play $d_2' = (2,2k+1)$. \St then replies with $s_3' = (1,2k-2)$, thereby forcing \Dom to choose $d_3' = (1,2k)$. \St now plays $s_4' = (2,2k-2)$, with the double threat of playing the vertex $(2,2k)$ and the vertex $(2,2k-3)$, thereby winning the S-game in $G - e$. Staller's winning strategy is illustrated in Figure~\ref{f:Swin1}(a) in the special case when $k = 7$.

\begin{figure}[ht!]
\begin{center}
\begin{tikzpicture}[scale=.8,style=thick,x=1cm,y=1cm]
\def\vr{2.5pt} 
\path (0,0) coordinate (v1);
\path (1,0) coordinate (v2);
\path (2,0) coordinate (v3);
\path (3,0) coordinate (v4);
\path (4,0) coordinate (v5);
\path (5,0) coordinate (v6);
\path (6,0) coordinate (v7);
\path (0,1) coordinate (u1);
\path (1,1) coordinate (u2);
\path (2,1) coordinate (u3);
\path (3,1) coordinate (u4);
\path (4,1) coordinate (u5);
\path (5,1) coordinate (u6);
\path (6,1) coordinate (u7);
\draw (v1) -- (v2);
\draw (v2) -- (v3);
\draw (v3) -- (v4);
\draw (v4) -- (v5);
\draw (v5) -- (v6);
\draw (v6) -- (v7);
\draw (u1) -- (u2);
\draw (u2) -- (u3);
\draw (u3) -- (u4);
\draw (u4) -- (u5);
\draw (u5) -- (u6);
\draw (u6) -- (u7);
\draw (v1) -- (u1);
\draw (v2) -- (u2);
\draw (v3) -- (u3);
\draw (v4) -- (u4);
\draw (v5) -- (u5);
\draw (v7) -- (u7);
\draw (v1) to[out=-50,in=-120, distance=1.75cm ] (v7);
\draw (u1) to[out=50,in=120, distance=1.75cm ] (u7);
\draw (v1) [fill=white] circle (\vr);
\draw (v2) [fill=white] circle (\vr);
\draw (v3) [fill=white] circle (\vr);
\draw (v4) [fill=white] circle (\vr);
\draw (v5) [fill=white] circle (\vr);
\draw (v6) [fill=white] circle (\vr);
\draw (v7) [fill=white] circle (\vr);
\draw (u1) [fill=white] circle (\vr);
\draw (u2) [fill=white] circle (\vr);
\draw (u3) [fill=white] circle (\vr);
\draw (u4) [fill=white] circle (\vr);
\draw (u5) [fill=white] circle (\vr);
\draw (u6) [fill=white] circle (\vr);
\draw (u7) [fill=white] circle (\vr);
\draw[anchor = west] (u7) node {{\small $d_2'$}};
\draw[anchor = west] (v7) node {{\small $d_1'$}};
\draw[anchor = north] (v6) node {{\small $d_3'$}};
\draw[anchor = south] (u5) node {{\small $s_2'$}};
\draw[anchor = north] (v5) node {{\small $s_1'$}};
\draw[anchor = south] (u4) node {{\small $s_4'$}};
\draw[anchor = north] (v4) node {{\small $s_3'$}};
\draw (3.5,-1.7) node {{\small (a)}};
\path (9,0) coordinate (v1);
\path (10,0) coordinate (v2);
\path (11,0) coordinate (v3);
\path (12,0) coordinate (v4);
\path (13,0) coordinate (v5);
\path (14,0) coordinate (v6);
\path (15,0) coordinate (v7);
\path (9,1) coordinate (u1);
\path (10,1) coordinate (u2);
\path (11,1) coordinate (u3);
\path (12,1) coordinate (u4);
\path (13,1) coordinate (u5);
\path (14,1) coordinate (u6);
\path (15,1) coordinate (u7);
\draw (v1) -- (v2);
\draw (v2) -- (v3);
\draw (v3) -- (v4);
\draw (v4) -- (v5);
\draw (v5) -- (v6);
\draw (v6) -- (v7);
\draw (u1) -- (u2);
\draw (u2) -- (u3);
\draw (u3) -- (u4);
\draw (u4) -- (u5);
\draw (u5) -- (u6);
\draw (u6) -- (u7);
\draw (v1) -- (u1);
\draw (v2) -- (u2);
\draw (v3) -- (u3);
\draw (v4) -- (u4);
\draw (v5) -- (u5);
\draw (v6) -- (u6);
\draw (v7) -- (u7);
\draw (u1) to[out=50,in=120, distance=1.75cm ] (u7);
\draw (v1) [fill=white] circle (\vr);
\draw (v2) [fill=white] circle (\vr);
\draw (v3) [fill=white] circle (\vr);
\draw (v4) [fill=white] circle (\vr);
\draw (v5) [fill=white] circle (\vr);
\draw (v6) [fill=white] circle (\vr);
\draw (v7) [fill=white] circle (\vr);
\draw (u1) [fill=white] circle (\vr);
\draw (u2) [fill=white] circle (\vr);
\draw (u3) [fill=white] circle (\vr);
\draw (u4) [fill=white] circle (\vr);
\draw (u5) [fill=white] circle (\vr);
\draw (u6) [fill=white] circle (\vr);
\draw (u7) [fill=white] circle (\vr);
\draw[anchor = south] (u1) node {{\small $d_1'$}};
\draw[anchor = south] (u3) node {{\small $d_2'$}};
\draw[anchor = north] (v2) node {{\small $s_1'$}};
\draw[anchor = north] (v4) node {{\small $s_2'$}};
\draw[anchor = north] (v6) node {{\small $s_3'$}};
\draw (12.5,-1.7) node {{\small (b)}};
\end{tikzpicture}
\end{center}
\vskip -0.5cm
\caption{Illustrating Staller's winning strategy in the proof of Proposition~\ref{p:Dminimal}} \label{f:Swin1}
\end{figure}
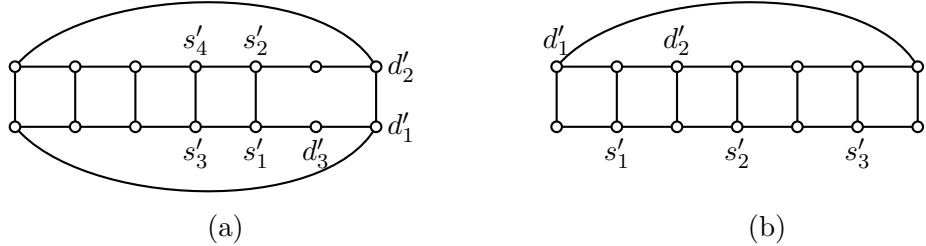

Suppose next that $e = (1,1)(1,2k+1)$.  Then \St starts the S-game with the move $s_1' = (1,2)$. If $k=1$, then this first move of \St has the double threat of playing the vertex $(2,1)$ and the vertex $(2,3)$. Hence, we may assume that $k \ge 2$, for otherwise \St immediately wins the S-game in $G - e$. With this assumption, \Dom is forced to play $d_1' = (2,1)$. \St then replies with $s_2' = (1,4)$. If $k = 2$, then this second move of \St has the double threat of playing the vertex $(2,3)$ and the vertex $(2,5)$. Hence, we may assume that $k \ge 3$, for otherwise \St wins the S-game in $G - e$. Continuing in this way, \St plays as her first $k-1$ moves the vertices $(1,2)$, $(1,4)$, \ldots, $(1,2k-2)$ in turn, thereby forcing \Dom to play as his first $k-1$ moves the vertices $(2,1)$, $(2,3)$, \ldots, $(2,2k-3)$ in turn. \St then plays $s_k' = (1,2k)$, with the double threat of playing the vertex $(2,2k-1)$ and the vertex $(2,2k+1)$,  thereby winning the S-game in $G - e$. Staller's winning strategy is illustrated in Figure~\ref{f:Swin1}(b) in the special case when $k = 7$.

Hence, in both cases $G - e$ is not $\cD$, implying that $P_2\cp C_{2k+1}$ is $\cD$-minimal.~\qed

\section{Cacti}
\label{sec:cacti}

A connected graph $G$ is a {\em cactus} if every block of $G$ is a cycle or $K_2$. An {\em end-block} of a (cactus) graph $G$ is a block of $G$ that intersects other blocks of $G$ in at most one vertex. In addition, $G$ is a {\em star cactus} if $G$ contains a vertex that is contained in each block of $G$. Equivalently, $G$ is a star cactus if every block of $G$ is an end-block. We have used the name star cactus because star cacti restricted to the class of trees are the stars $K_{1,n}$.

In this section we classify cacti with respect to the classes ${\cal D}$, ${\cal S}$, and ${\cal N}$ and begin with a sequence of lemmata.

\begin{lemma}
\label{lem:star-cactus}
Let $G$ be a star cactus with at least two blocks. Then the MBTD game is ${\cal N}$ if and only if the cycles of $G$ are of length at most $5$ and $G$ contains at least one of the following: a $3$-cycle, a $4$-cycle, or two $K_2$ blocks. Otherwise, the game is ${\cal S}$.
\end{lemma}

\proof
Let $u$ be the vertex of the star cactus $G$ that lies in all the blocks of $G$.

Consider first the S-game played on $G$. Then Staller plays $u$ as her first move. If $G$ contains a $K_2$ block, Staller already wins with this move. If not, then $u$ is contained in at least two cycles. Hence, after the first move of Dominator there is at least one cycle in which Dominator did not play and on this cycle Staller can isolate a neighbor of $u$. Hence, if the S-game is played on a star cactus with at least two blocks, then Staller has a winning strategy.

In the rest we consider the D-game and distinguish the following cases.

\medskip\noindent
{\bf Case 1}: $G$ contains a cycle of length at least $6$.\\
Let $C$ be such a cycle. Assume first that Dominator either plays $u$ as his first move or a vertex not on $C$. Then Staller plays a vertex $v$ of $C$ with $d_C(u,v) = 3$, thus posing a double threat on the neighbors of $v$ which will enable Staller to win after her subsequent move. Suppose next that Dominator starts the game by playing a vertex $w$ of $C$, $w\ne u$. Then Staller replies with a move on $u$. If $G$ contains a $K_2$ block, Staller wins with this move. Hence assume that there is another cycle $C'$ in $G$. If Dominator replies to the move $u$ of Staller with a move on $C$, then Staller will win by isolating a neighbor of $u$ on $C'$. Otherwise, at least one of the vertices on $C$ at distance $2$ from $u$, say $x$, is not played by Dominator and hence Staller wins by playing $x$ since then the common neighbor of $x$ and $u$ becomes isolated.

\medskip\noindent
{\bf Case 2}: $G$ contains only cycles of length $5$ and at most one $K_2$ block.\\
If Dominator does not start the game on $u$, then Staller can apply the above strategy to win the game. Suppose hence that Dominator first plays $u$. Let $C$ be an arbitrary $C_5$ block. (Note that such a block exists as $G$ has at least two blocks.) Let the vertices of $C$ be $u_1 = u, u_2, u_3, u_4, u_5$ in the natural order. Then Staller plays $u_2$ threatening $u_3$. Dominator has to play $u_4$ for otherwise Staller wins. Then Staller plays $u_5$ threatening $u_4$, and then Dominator must play $u_3$. Note that Staller played on both neighbors of $u$ in $C$, and will play the next move in some other block. So she can apply the same strategy for every $C_5$ block. Afterwards, if there is no $K_2$ block, then $u$ becomes isolated. Otherwise Staller plays the leaf of the unique $K_2$ block, isolating $u$ again.

\medskip\noindent
{\bf Case 3}: $G$ contains only cycles of length at most $5$, and either at least two $K_2$ blocks, or at least one $C_3$ block, or at least one $C_4$ block.\\
In this case we are going to prove that Dominator has a winning strategy. To do so, he first plays on $u$. Let $C$ be an arbitrary $C_5$ with its vertices $u_1 = u, u_2, u_3, u_4, u_5$. Then the strategy of Dominator on $C$ is that if Staller plays a vertex from $\{u_2,u_4\}$, then Dominator answers with the other vertex from the pair, and does the same for the pair $\{u_3,u_5\}$. Note that applying this strategy Dominator ensures that all the vertices of $V(C)\setminus\{u\}$ are totally dominated. Dominator applies this strategy on every $C_5$ block. In addition, if Staller plays a vertex of a $C_4$ or a $C_3$ block, then Dominator can reply on one of the two neighbors of $u$ in the same block. Doing so, he totally dominates the whole block (including $u$). Suppose finally that there are no $C_3$ or $C_4$ blocks. Then $G$ contains at least two $K_2$ blocks. The leaves of these blocks are already totally dominated by the first move of Dominator. Whatever Staller does, Dominator can totally dominate $u$ by playing one of these leaves.   In summary, every vertex of $G$ will be totally dominated by Dominator's moves.
\qed

\begin{lemma}
\label{lem:remove-C4}
If $C=C_4$ is an end-block of a connected graph $G$, then the outcome of the S-game on $G$ is the same as on $G\setminus C$.
\end{lemma}

\proof
If $G=C_4$, then the assertion holds because both $C_4$ and $G\setminus C_4 = \emptyset$ are ${\cal D}$ graphs. We may thus assume in the rest that $C$ contains a (unique) vertex of degree more than $2$. Set $G' = G\setminus C$. Suppose first that Dominator has a winning strategy for the S-game on $G'$. Then the strategy of Dominator is to follow Staller on $C$ as well as on $G'$ using his winning strategies, gives a winning strategy of Dominator on $G$. In the case that Staller has a winning strategy on $G'$, then she starts the S-game on $G$ by playing $u$. This forces Dominator to play the vertex of $C$ opposite to $u$. Afterwards Staller can follow her optimal strategy on $G'$ to win the game on $G$ as well. This is possible since $C$ is separated from $G'$ after the first move of Staller.
\qed

Note that Lemma~\ref{lem:remove-C4} remains valid if some moves of the game were already played, it is Staller's turn, and no vertex of $C_4$ has already been played.

\begin{lemma}
\label{lem:no-pendant-C4}
If $G$ is a non-empty cactus that contains no end-block $C_4$, then Staller has a winning strategy in the S-game on $G$.
\end{lemma}

\proof
If $G=K_1$ the assertion is clear. Suppose next that $\delta(G) = 1$. Then Staller plays the support vertex of a leaf to win the game. The last case is when $\delta(G)\ge 2$ which is equivalent to the fact that every end-block of $G$ is a cycle. By the assumption, none of these cycles is $C_4$. If $G=C_n$, $n\ne 4$, then by Proposition~\ref{prp:cycles} Staller wins. Assume finally that $G$ has more than one cycle and let $C$ be such an end-cycle with $u$ the (unique) vertex of $C$ of degree more than $2$. Then Staller plays $u$ as her first move and in this way makes a double threat on its neighbors on the cycle. Hence, Staller wins again.
\qed

Recall from Proposition~\ref{prp:cycles} that $C_3$ is the only cactus graph with a single block that is an ${\cal N}$ graph. Hence, in view of Lemma~\ref{lem:star-cactus}, we say that a cactus graph $G$ is an {\em ${\cal N}$-star cactus} if either $G=C_3$, or $G$ has at least two blocks, cycles are of length at most $5$, and contains a $3$-cycle, a $4$-cycle, or two $K_2$ blocks. The main result of this section now reads as follows.

\begin{theorem}
\label{thm:cactus}
Let $G$ be a cactus with at least two blocks. Then
\begin{enumerate}
\item[(i)] $G$ is ${\cal D}$ if and only if $V(G)$ can be partitioned into $4$-sets, each inducing a $C_4$;
\item[(ii)] $G$ is ${\cal N}$ if and only if there exists a sequence of end-blocks $C_4$ such that iteratively removing them yields an ${\cal N}$-star cactus.
\end{enumerate}
Consequently, $G$ is ${\cal S}$ in all the other cases.
\end{theorem}

\proof
(i) If $V(G)$ can be partitioned into $4$-sets, each inducing a $C_4$, then $G$ is a ${\cal D}$ graph by Corollary~\ref{cor:win-in-the-other-game}(ii). Conversely, assume that $V(G)$ can not be covered by vertices of disjoint $C_4$. Let $H$ be a graph obtained by iteratively removing end-blocks $C_4$ of $G$. Then $H$ is not the empty graph and by Lemma~\ref{lem:remove-C4}, the outcome of the S-game on $H$ is the same as the outcome on $G$. But then Staller wins the game on $H$ by Lemma~\ref{lem:no-pendant-C4}.

(ii) Assume first that there exists a sequence of end-blocks $C_4$ such that iteratively removing them yields an ${\cal N}$-star cactus, denote it with $H$. Then Dominator considers the game as to be played on the disjoint union of several $C_4$s and $H$. Each $C_4$ is ${\cal D}$ and $H$ is ${\cal N}$ by Proposition~\ref{prp:cycles} and Lemma~\ref{lem:star-cactus}. Hence by Corollary~\ref{cor:win-in-the-other-game}(ii), Dominator can win playing first. On the other hand, if Staller plays first, then by Lemma~\ref{lem:remove-C4}, the outcome of the S-game on $H$ is the same as on $G$.  Hence Staller wins on $G$.

Conversely, suppose that there does not exist a sequence of end-blocks $C_4$ such that iteratively removing them yields an ${\cal N}$-star cactus. Consider the D-game played on $G$ and let $u$ be the first move of Dominator. Let $H$ be a graph obtained from $G$ by iteratively removing pendant blocks $C_4$ that do not contain $u$ until no such end-block remains. Note that $H\ne C_4$, for otherwise $G$ can be covered with disjoint $C_4$s and we are in (i). Then the outcome of the game on $H$ is the same as on $G$ by Lemma~\ref{lem:remove-C4}. If $H$ is not a star cactus, then $H$ contains at least two disjoint end-blocks. One of these blocks does not contain $u$ and this block, say $B$, is not a $C_4$. Staller can play on the vertex $x$ of highest degree in $B$. If $B=K_2$ then Staller wins right away, otherwise she threatens both neighbors of $x$ in $B$, so she can win after the next move of Dominator. Suppose next that $H$ is a star cactus. Then it is neither an ${\cal N}$-star cactus nor $C_4$. Now, whatever Dominator plays as his first move, Staller has a winning strategy. Hence $G$ is an ${\cal S}$ graph.
\qed

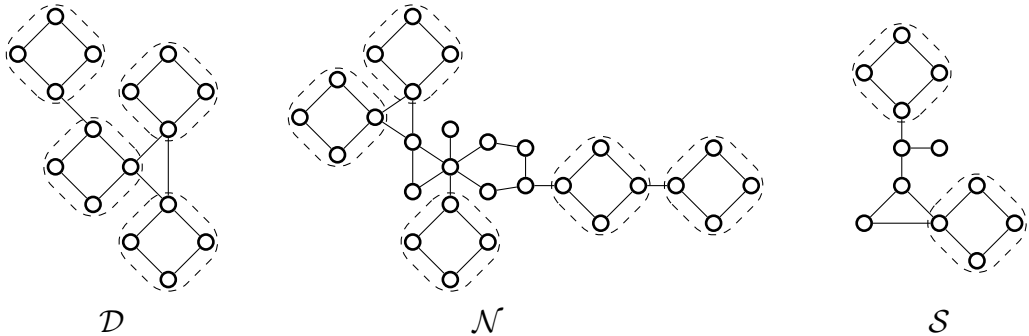
\begin{figure}[ht!]
\begin{center}
\begin{tikzpicture}

\node at (0,0){
\begin{tikzpicture}
    \node[noeud] (a1) at (0,0){};
    \node[noeud] (b1) at (-0.5,0.5){};
    \node[noeud] (c1) at (-1,0){};
    \node[noeud] (d1) at (-0.5,-0.5){};
    \node[noeud] (a2) at (-1,1){};
    \node[noeud] (b2) at (-1.5,1.5){};
    \node[noeud] (c2) at (-1,2){};
    \node[noeud] (d2) at (-0.5,1.5){};
    \node[noeud] (a3) at (0.5,0.5){};
    \node[noeud] (b3) at (0,1){};
    \node[noeud] (c3) at (0.5,1.5){};
    \node[noeud] (d3) at (1,1){};
    \node[noeud] (a4) at (0.5,-0.5){};
    \node[noeud] (b4) at (0,-1){};
    \node[noeud] (c4) at (0.5,-1.5){};
    \node[noeud] (d4) at (1,-1){};

\draw (a1)-- (b1) -- (c1) -- (d1) -- (a1);
\draw (a2)-- (b2) -- (c2) -- (d2) -- (a2);
\draw (a3)-- (b3) -- (c3) -- (d3) -- (a3);
\draw (a4)-- (b4) -- (c4) -- (d4) -- (a4);
\draw (a1) -- (a3);
\draw (a1) -- (a4);
\draw (a4) -- (a3);
\draw (b1) -- (a2);

\draw[dashed, rounded corners=10pt] (0.3,0) -- (-0.5,0.8) -- (-1.3,0) -- (-0.5, -0.8) -- cycle;
\draw[dashed, rounded corners=10pt] (-1,0.7) -- (-1.8,1.5) -- (-1,2.3) -- (-0.2, 1.5) -- cycle;
\draw[dashed, rounded corners=10pt] (0.5,0.2) -- (-0.3,1) -- (0.5,1.8) -- (1.3, 1) -- cycle;
\draw[dashed, rounded corners=10pt] (0.5,-0.2) -- (-0.3,-1) -- (0.5,-1.8) -- (1.3, -1) -- cycle;
\end{tikzpicture}
};

\node at (0,-2.3){$\cal D$};

\node at (5.5,0){
\begin{tikzpicture}
    \node[noeud] (cent) at (0,0){};
    \node[noeud] (a) at (0,0.5){};
    \node[noeud] (b) at (-0.5,0.33){};
    \node[noeud] (c) at (-0.5,-0.33){};
    \node[noeud] (d) at (0,-0.5){};
    \node[noeud] (e) at (0.5,0.33){};
    \node[noeud] (f) at (0.5,-0.33){};
    \node[noeud] (t1) at (-1,0.66){};
    \node[noeud] (t2) at (-0.5,1){};
    \node[noeud] (h1) at (1,0.25){};
    \node[noeud] (h2) at (1,-0.25){};
    \node[noeud] (v1) at (-1.5,1.16){};
    \node[noeud] (v2) at (-2,0.66){};
    \node[noeud] (v3) at (-1.5,0.16){};
    \node[noeud] (w1) at (0,1.5){};
    \node[noeud] (w2) at (-0.5,2){};
    \node[noeud] (w3) at (-1,1.5){};
    \node[noeud] (x1) at (-0.5,-1){};
    \node[noeud] (x2) at (0,-1.5){};
    \node[noeud] (x3) at (0.5,-1){};
    \node[noeud] (y1) at (1.5,-0.25){};
    \node[noeud] (y2) at (2,-0.75){};
    \node[noeud] (y3) at (2.5,-0.25){};
    \node[noeud] (y4) at (2,0.25){};
    \node[noeud] (z1) at (3,-0.25){};
    \node[noeud] (z2) at (3.5,-0.75){};
    \node[noeud] (z3) at (4,-0.25){};
    \node[noeud] (z4) at (3.5,0.25){};

\draw (cent) -- (a);
\draw (cent) -- (b);
\draw (cent) -- (c);
\draw (cent) -- (d);
\draw (cent) -- (e);
\draw (cent) -- (f);
\draw (b)--(c);
\draw (b)--(t1);
\draw (b)--(t2);
\draw (t1)--(t2);
\draw (f)--(h2)--(h1)--(e);
\draw (t1)--(v1)--(v2)--(v3)--(t1);
\draw (t2)--(w1)--(w2)--(w3)--(t2);
\draw (d)--(x1)--(x2)--(x3)--(d);
\draw (y1)--(y2)--(y3)--(y4)--(y1);
\draw (z1)--(z2)--(z3)--(z4)--(z1);
\draw (h2) -- (y1);
\draw (y3) -- (z1);

\draw[dashed, rounded corners=10pt] (-0.7,0.66) -- (-1.5,1.46) -- (-2.3,0.66) -- (-1.5,-0.14) -- cycle;
\draw[dashed, rounded corners=10pt] (-0.5,0.7) -- (0.3,1.5) -- (-0.5,2.3) -- (-1.3, 1.5) -- cycle;
\draw[dashed, rounded corners=10pt] (0,-0.2) -- (-0.8,-1) -- (0,-1.8) -- (0.8, -1) -- cycle;
\draw[dashed, rounded corners=10pt] (1.2,-0.25) -- (2,-1.05) -- (2.8,-0.25) -- (2,0.55) -- cycle;
\draw[dashed, rounded corners=10pt] (2.7,-0.25) -- (3.5,-1.05) -- (4.3,-0.25) -- (3.5,0.55) -- cycle;

\end{tikzpicture}
};

\node at (5,-2.3){$\cal N$};

\node at (11,0){
\begin{tikzpicture}

    \node[noeud] (a) at (0,0){};
    \node[noeud] (b) at (0.5,0){};
    \node[noeud] (c) at (0,-0.5){};
    \node[noeud] (d) at (-0.5,-1){};
    \node[noeud] (f1) at (0.5,-1){};
    \node[noeud] (f2) at (1,-0.5){};
    \node[noeud] (f3) at (1.5,-1){};
    \node[noeud] (f4) at (1,-1.5){};
    \node[noeud] (e1) at (0,0.5){};
    \node[noeud] (e2) at (0.5,1){};
    \node[noeud] (e3) at (0,1.5){};
    \node[noeud] (e4) at (-0.5,1){};

\draw (e1) -- (e2) -- (e3) -- (e4) -- (e1) -- (a) -- (c) -- (d) -- (f1)	-- (f2) -- (f3) -- (f4) -- (f1) -- (c);
\draw (a) -- (b);

\draw[dashed, rounded corners=10pt] (0,0.2) -- (0.8,1) -- (0,1.8) -- (-0.8, 1) -- cycle;
\draw[dashed, rounded corners=10pt] (0.2,-1) -- (1,-0.2) -- (1.8,-1) -- (1,-1.8) -- cycle;
\end{tikzpicture}
};

\node at (11,-2.3){$\cal S$};

\end{tikzpicture}
\end{center}
\caption{Examples for each type of cacti}
\label{fig:ex_cacti}
\end{figure}

In~\cite{gledel-2018+} it is proved that if $T$ is a tree, then the MBD game on $T$ is ${\cal D}$ if $T$ has a perfect matching, it is ${\cal N}$ if by iteratively removing pendant $P_2$ from $T$ a star is obtained, and it is ${\cal S}$ otherwise.  Hence Theorem~\ref{thm:cactus} is a result parallel to this, where $4$-cycles play the role of $P_2$s and ${\cal N}$-star cacti the role of stars. It is also interesting to note that in the case of the MBD game $P_2$ is the smallest ${\cal D}$ graph while $C_4$ is the smallest ${\mathcal D}$ graph for the MBTD game.

For trees, Theorem~\ref{thm:cactus} reduces to:

\begin{corollary}
\label{cor:trees}
If $T$ is a tree, then the MBTD game is ${\cal N}$ if $T = K_{1,n}$, $n\ge 2$, otherwise the game is ${\cal S}$.
\end{corollary}

Note that Lemma~\ref{lem_MBD-versus-MBTD} applied to Theorem~\ref{thm:cactus} yields some new insight into the MBD game played on cacti.


\section{Complexity Results}
\label{S:complexity}

In this section we prove that the problem of deciding whether a given graph $G$ is $\mathcal D$, $\mathcal S$, or $\mathcal N$ for the MBTD game is PSPACE-complete. As in the case of the parallel results for the MBD game from~\cite{gledel-2018+}, our proof uses a reduction from the POS-CNF game. This game is the two player game played on a Conjunctive Normal Form (CNF) composed of variables, $x_1, \ldots, x_n$, and of clauses $C_1, \ldots, C_m$, where all variables appear only positively. In this game, the first player, Prover, assigns variables to True and wins if the formula evaluates to True. The second player, Disprover, assigns variables to False and wins if the formula evaluates to False. Schaefer proved in 1978 that this game is PSPACE-complete~\cite{schaefer-1978}.

Recall that a graph $G$ is {\em split} if $V(G)$ can be partitioned into two sets, one inducing a clique and the other an independent set.

\begin{theorem}\label{thm:pspace_split}
Deciding the outcome of the MBTD game is PSPACE-complete on split graphs.
\end{theorem}

\proof
As the MBTD game is a combinatorial game which ends after a finite number of moves, it is in PSPACE. We will now prove that it is PSPACE-hard.

Let $(x_i)_{1 \leq i \leq n}$, $(C_j)_{1 \leq j \leq m}$, be an instance of POS-CNF. Let $G$ be the split graph on the set of vertices $V=\{u_i:\ 1 \leq i \leq n\} \cup \{v_j:\ 1 \leq j \leq m\}$, where the vertices $u_i$ form a clique, the vertices $v_j$ form an independent set and two vertices $u_i$ and $v_j$ form an edge if and only if $x_i$ is a variable of the clause $C_j$. The obtained graph is clearly a split graph. Figure~\ref{fig:pspace}(a) illustrates an example of this construction.

We will now prove that \Dom wins the MBTD game on $G$ if and only if Prover wins the POS-CNF game. Assume that Prover has a winning strategy on the POS-CNF game. In this case, \Dom can win the MBTD game by using the following strategy. Each time Prover's assigns a variable $x_i$ to True, \Dom plays the vertex $u_i$. Each time \St plays on a vertex $u_{i'}$, \Dom plays as if Disprover assigned the vertex $x_{i'}$ to False. If \St plays on a vertex $v_j$, then \Dom plays as if she played on an arbitrary vertex $u_{i''}$.
Following his strategy for the POS-CNF game, Prover is able to satisfy all the clauses, so by imitating his strategy \Dom\ is able to totally dominate the vertices $v_j$. Since the vertices $u_i$ form a clique, playing twice in the clique totally dominates it, and so \Dom\ has a winning strategy on $G$. (If $n < 4$, we can assume that there are $4-n$ more variables that don't appear in any clauses and this does not change the outcome of the game.)

Assume now that Disprover has a winning strategy on the POS-CNF game. We will demonstrate that in this case, the following is a winning strategy for \St on the MBTD game on $G$. Each time Disprover assigns a variable $x_i$ to False, \St plays on the vertex $u_i$. Each time \Dom plays on a vertex $u_{i'}$, she follows Disprover's strategy in the case where Prover assigned the variable $x_{i'}$ to True. And each time \Dom plays on a vertex $v_j$, \St plays as if he played on an arbitrary vertex $u_{i''}$. Since Disprover has a winning strategy, she can assign each variable of some clause $C_j$ to False and, by imitating this strategy, \St can play on every neighbour of the vertex $v_j$, thus keeping \Dom from totally dominating it.

Note that these strategies work both in the case when Prover starts and in the case when Disprover starts. 
\qed

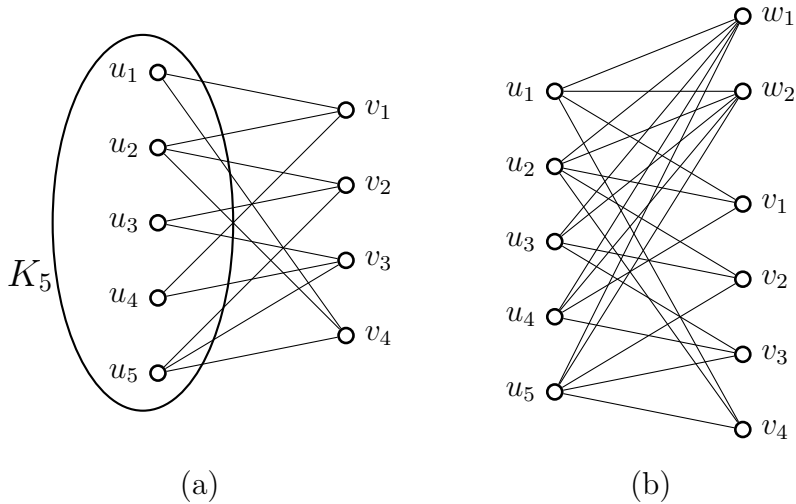
\begin{figure}[ht!]
	\centering

	\begin{tikzpicture}
	\node at (-3,0){
		\begin{tikzpicture}
		\node[noeud] (x1) at (0,0){};
		\node[noeud] (x2) at (0,-1){};
		\node[noeud] (x3) at (0,-2){};
		\node[noeud] (x4) at (0,-3){};
		\node[noeud] (x5) at (0,-4){};
		
		\node[left=3pt] at (x1){$u_1$};
		\node[left=3pt] at (x2){$u_2$};
		\node[left=3pt] at (x3){$u_3$};
		\node[left=3pt] at (x4){$u_4$};
		\node[left=3pt] at (x5){$u_5$};
		
		\draw[thick] (-0.2,-2) ellipse (1.2cm and 2.5cm);
		
		\node at (-1.7,-2.7){\large $K_5$};
		
		\node[noeud] (c1) at (2.5,-0.5){};
		\node[noeud] (c2) at (2.5,-1.5){};
		\node[noeud] (c3) at (2.5,-2.5){};
		\node[noeud] (c4) at (2.5,-3.5){};

		\node[right=3pt] at (c1){$v_1$};
		\node[right=3pt] at (c2){$v_2$};
		\node[right=3pt] at (c3){$v_3$};
		\node[right=3pt] at (c4){$v_4$};
		
		\draw (c1) -- (x1);
		\draw (c1) -- (x2);
		\draw (c1) -- (x4);
		
		\draw (c2) -- (x2);
		\draw (c2) -- (x3);
		\draw (c2) -- (x5);
		
		\draw (c3) -- (x3);
		\draw (c3) -- (x4);
		\draw (c3) -- (x5);
		
		\draw (c4) -- (x1);
		\draw (c4) -- (x2);
		\draw (c4) -- (x5);		
		\end{tikzpicture}
		};
		
	\node at (3,0){
		\begin{tikzpicture}
		\node[noeud] (x1) at (0,0){};
		\node[noeud] (x2) at (0,-1){};
		\node[noeud] (x3) at (0,-2){};
		\node[noeud] (x4) at (0,-3){};
		\node[noeud] (x5) at (0,-4){};
		
		\node[left=3pt] at (x1){$u_1$};
		\node[left=3pt] at (x2){$u_2$};
		\node[left=3pt] at (x3){$u_3$};
		\node[left=3pt] at (x4){$u_4$};
		\node[left=3pt] at (x5){$u_5$};

		\node[noeud] (c1) at (2.5,-1.5){};
		\node[noeud] (c2) at (2.5,-2.5){};
		\node[noeud] (c3) at (2.5,-3.5){};
		\node[noeud] (c4) at (2.5,-4.5){};

		\node[noeud] (w2) at (2.5,0){};
		\node[noeud] (w1) at (2.5,1){};

		\node[right=3pt] at (c1){$v_1$};
		\node[right=3pt] at (c2){$v_2$};
		\node[right=3pt] at (c3){$v_3$};
		\node[right=3pt] at (c4){$v_4$};
		
		\node[right=3pt] at (w1){$w_1$};
		\node[right=3pt] at (w2){$w_2$};
		
		\draw (c1) -- (x1);
		\draw (c1) -- (x2);
		\draw (c1) -- (x4);
		
		\draw (c2) -- (x2);
		\draw (c2) -- (x3);
		\draw (c2) -- (x5);
		
		\draw (c3) -- (x3);
		\draw (c3) -- (x4);
		\draw (c3) -- (x5);
		
		\draw (c4) -- (x1);
		\draw (c4) -- (x2);
		\draw (c4) -- (x5);	
		
		\draw (w1) -- (x1);
		\draw (w1) -- (x2);
		\draw (w1) -- (x3);
		\draw (w1) -- (x4);	
		\draw (w1) -- (x5);
		
		\draw (w2) -- (x1);
		\draw (w2) -- (x2);
		\draw (w2) -- (x3);
		\draw (w2) -- (x4);	
		\draw (w2) -- (x5);
		\end{tikzpicture}
		};

		\node at (-3,-3.5) {(a)};
		\node at (3,-3.5) {(b)};		
		
		\node at (0,-4.5) {$F=(x_1 \vee x_2 \vee x_4) \wedge (x_2 \vee x_3 \vee x_5) \wedge (x_3 \vee x_4 \vee x_5) \wedge (x_1 \vee x_2 \vee x_5)$};		
		
	\end{tikzpicture}
	\caption{The graph corresponding to the formula $F$ following the constructions of the proof of Theorem~\ref{thm:pspace_split} and of the proof of Corollary~\ref{cor:pspace_bipartite}.}
	\label{fig:pspace}
\end{figure}

\begin{corollary}
\label{cor:pspace_bipartite}
Deciding the outcome of the MBTD game is PSPACE-complete on bipartite graphs.
\end{corollary}

\begin{proof}
We are going to accordingly adopt the proof of Theorem~\ref{thm:pspace_split}. For this purpose, it is sufficient to only alter a little the previous construction and strategies. Instead of joining the vertices $u_i$ into a clique, they now form an independent set. We next add two new vertices, $w$ and $w'$, joined by an edge to every vertex $u_i$. An example of this construction is illustrated in Figure~\ref{fig:pspace}(b).

The resulting graph is bipartite and we can modify Dominator's strategy so that if Prover has a winning strategy for the POS-CNF game, then \Dom has a winning strategy for the MBTD game. If at some point \St plays on $w$ or $w'$, then \Dom answers by playing on the other vertex. Since $w$ and $w'$ are adjacent to all the vertices $u_i$, they will be totally dominated when \Dom plays on one of those vertices and all of the vertices $u_i$ will either be totally dominated by $w$ or $w'$. The situation of the vertices $v_j$ is similar to what it was in the previous proof. Staller's strategy in the case when Disprover wins on the POS-CNF game is the same as before, and she can ignore the case when \Dom plays on $w$ or $w'$ as she does when he plays on one of the vertices $v_j$. 
\end{proof}

\section{Concluding remarks}
\label{sec:concluding}

In this paper, we have introduced and studied the total version of the Maker-Breaker domination game. We close with the following open problems.

\begin{problem}
Characterize connected cubic graphs that are ${\cal D}$ and those that are ${\cal S}$.
\end{problem}

\begin{question}
Is it true that Dominator wins the D-game on an arbitrary $k$-regular graph, $k\ge 4$, if the girth of $G$ is small?
\end{question}

\begin{question}
Is it true that $P_{2k+1}\cp C_{2\ell +1}$ is ${\cal S}$ for every $k, \ell \ge 1$?
\end{question}

\begin{problem}
Characterize ${\cal D}$-minimal graphs. In particular, find additional families of ${\cal D}$-minimal graphs.
\end{problem}

\begin{problem}
Is there a graph of girth at least $5$ that is ${\cal D}$?
\end{problem}


\end{document}